\documentclass[12pt]{article}

\usepackage[T1]{fontenc}
\usepackage[hmargin=1.0in]{geometry}

\usepackage{graphicx}
\usepackage{hyperref}
\usepackage{amsmath,amssymb,amsthm}
\usepackage{booktabs}
\usepackage{todonotes}
\usepackage[caption=false]{subfig}

\newtheorem{remark}{Remark}
\theoremstyle{plain}
\newtheorem{theorem}{Theorem}
\newtheorem{lemma}{Lemma}

\newcommand{\sspcoeff}{{\mathcal C}}
\newcommand{\Real}{{\mathbb R}}
\newcommand{\ee}{\mathbf{e}}

\newcommand{\bb}{\overline{\textbf{\textit{b}}}}
\newcommand{\betab}{\overline{\beta}}

\begin{document}

\title{Dense output for strong stability preserving Runge--Kutta methods}

\author{
David I. Ketcheson\thanks{Corresponding author. King Abdullah University of Science and Technology (KAUST), Thuwal 23955-6900, Saudi Arabia (david.ketcheson@kaust.edu.sa).} \and 
Lajos L{\'o}czi\thanks{King Abdullah University of Science and Technology (KAUST), Thuwal 23955-6900, Saudi Arabia (lajos.loczi@kaust.edu.sa).} \and
Aliya Jangabylova\thanks{Nazarbayev University, Kazakhstan.} \and 
    Adil Kusmanov\thanks{Nazarbayev University, Kazakhstan.}
\thanks{This work was supported by the King Abdullah University of Science and Technology (KAUST), 4700 Thuwal, 23955-6900, Saudi Arabia.  The last two authors were supported by the KAUST Visiting Student Research Program.}}

\maketitle

\begin{abstract}
We investigate dense output formulae (also known as continuous extensions)
for strong stability preserving (SSP) Runge--Kutta  methods.  We require 
that the dense output formula also possess the SSP property, ideally under
the same step-size restriction as the method itself.  A general recipe
for first-order SSP dense output formulae for SSP methods is given, and
second-order dense output formulae for several optimal SSP methods are developed.
It is shown that  SSP dense output formulae of order 3 and higher 
do not exist, and that in any 
method possessing a second-order SSP dense output, the coefficient matrix
$A$ has a zero row.
\end{abstract}

\section{Motivation and goals}
Strong stability preserving (SSP) Runge--Kutta (RK) methods are 
widely used in the time discretization of hyperbolic conservation
laws.  SSP methods guarantee the preservation of solution properties
based on convex functionals, such as monotonicity
or contractivity (in arbitrary norms); positivity; and maximum
principles. For a review of SSP methods see \cite{2009_ssp_review,SSPbook}.

An SSP integrator is chosen usually in order to maximize the allowed
step size while guaranteeing some desired property.
In practice it frequently happens that the next time step would go beyond a desired output time.  In this case, it is common to shorten the step in order to reach the output time exactly.  If very frequent output is required, 
as is often the case in aeroacoustics applications, for instance,
then the numerical steps may be much smaller than what the 
method would otherwise allow.  In particular, 
for SSP methods frequent output may mean that the steps used in practice are
much smaller than the SSP step size, which is inefficient.
Even if the spacing between output times is much larger than the
SSP step size, it may be necessary to use a small step 
immediately before each output time.  In the context of hyperbolic
problems, this corresponds to a step with small CFL number,
which introduces a larger amount of dissipation.

In order to avoid the need to stop exactly at the output times,
a RK method can be designed with a {\em dense output formula},
also known as a {\em continuous extension} 
\cite{zennaro1986natural,Enright:1986:IRF:7921.7923}.  Whereas the RK 
method provides output only at the discrete step times 
$t_1, t_2, \dots$,
the dense output formula can be used to obtain output at
arbitrary times, often without the need for any additional 
function evaluations.
The construction of dense output formulae for general
RK methods is well understood.  For instance,
dense output of at least third order accuracy can be 
obtained for any RK method using Hermite interpolation,
and higher-order formulae can also be derived \cite[Section 2.6]{hairer1993}.  

Contractivity of continuous RK methods has been studied previously
in \cite{torelli1991sufficient,bellen1997contractivity} and related
works.  Therein, the interest was primarily in unconditional
contractivity for applications to delay differential equations, and
a diagonally split Runge--Kutta method was shown to achieve
unconditional contractivity (see also \cite{macdonald2008numerical}).

In the present work, we investigate dense output formulae for explicit SSP
RK methods.  Naturally, we require that the dense output formula
also possess the SSP property, under the same (finite) step-size restriction
as the overall method.  This turns out to be a very strong
requirement:
we prove that for any SSP RK method, there exists no SSP 
dense output formula of order three or higher.
On the other hand, we show that for many SSP methods, SSP dense output of order
one or two can be constructed in a simple way.

\subsection{A numerical example}
Let us motivate this work by showing what can go wrong if no attention
is paid to the SSP property of a dense output formula.

As a test problem, we choose the scalar ODE
\begin{subequations} \label{sinode}
\begin{align} 
    u'(t) & = f(t,u(t)):=\sin(10t)\,u(t)\, (1-u(t))\\
    u(0) &=u_0
\end{align}
\end{subequations}
with various initial conditions $u_0\in I_1:=[0,1]$. 
It is easily seen that the exact solution of \eqref{sinode}
\[
u(t)={u_0}\left(u_0  +\left(1-u_0\right) \exp\left(\frac{\cos (10
   t)-1}{10}\right)\right)^{-1}
\]
remains in $I_1$ for any $u_0\in I_1$ and $t\ge 0$.
It is also straightforward to show that applying the forward Euler method 
\[
    u_{n+1} = u_n+h f(n h,u_n)
\]
to \eqref{sinode} with an arbitrary starting value $u_0\in I_1$ and fixed step size $h\in 
[0,h_{\textrm{FE}}]$ with $h_{\textrm{FE}}:=1$ yields a solution $u_{n+1} \in I_1$.

Now let us apply the following 3-stage, 2nd-order SSP Runge--Kutta method to
the same problem:
\begin{subequations}\label{numexamplemethod}
\begin{align}
    y_1 & = u_n \\
    y_2 & = y_1 + \frac{h}{2} f(t_n, y_1) \\
    y_3 & = y_2 + \frac{h}{2} f(t_n + h/2, y_2) \\
    u_{n+1} & = \frac{1}{3} u_n + \frac{2}{3} \left(y_3 + \frac{h}{2} f(t_n + h, y_3)\right).
\end{align}
\end{subequations}
Since each stage of this method is a convex combination of Euler steps with
step size at most $h/2$, it also preserves the invariance of $I_1$, under the
larger step-size restriction $h\in[0,2h_{\textrm{FE}}]=[0,2]$.

In order to evaluate the solution at off-step points, we will consider two dense output formulae.
Here $u_{n+\theta}$ is an approximation to $u(t_n + \theta h)$, and we define the
solution in a piecewise fashion so that $\theta \in [0,1]$.
The first candidate formula is
\begin{align} \label{SSP}
    u_{n+\theta} & = \left(1-2\theta + \frac{4}{3}\theta^2\right)u_n + 
        2(\theta-\theta^2)\left(y_1+\frac{h}{2}f(t_n,y_1)\right) +
        \frac{2}{3}\theta^2\left(y_3 + \frac{h}{2}f(t_n+h,y_3)\right).
\end{align}
The second candidate formula is
\begin{align} \label{nonSSP}
    u_{n+\theta} = u_n + h\left( (2\theta-\theta^2)f(t_n,y_1) + (-2\theta + \theta^2)f(t_n+h/2,y_2) + \theta f(t_n+h,y_3)\right).
\end{align}
Both \eqref{SSP} and \eqref{nonSSP} are 2nd-order accurate.
Regarding formula \eqref{SSP},
since the coefficient functions $1-2\theta + 4\theta^2/3$, 
$2(\theta-\theta^2)$ and $2\theta^2/3$ are non-negative for $\theta\in[0,1]$, this
formula is also a convex combination of Euler steps, with step size
$h/2$.  So it also preserves the invariance of $I_1$ for $h\in[0,2h_{\textrm{FE}}]$.
This is however not true of formula \eqref{nonSSP}.


We integrate \eqref{sinode} with method \eqref{numexamplemethod} for a range of 
initial conditions $u_0 \in I_1$ with step size $h=1.6$
to obtain the values given by the set of black dots in Figure \ref{fig:sinode} 
(the two sets of dots are the same in both subfigures). 
The application of the two dense output 
formulae \eqref{SSP} and \eqref{nonSSP} to \eqref{sinode} with the same
step size $h=1.6$ now
provides us with two sets of piecewise quadratic interpolants.
As shown in Figure
\ref{fig:sinode}, the SSP dense output \eqref{SSP} preserves the relation $u_{n+\theta}\in I_1$ 
while the non-SSP dense output \eqref{nonSSP} does not. 
Tests with larger step sizes confirm that dense output using the SSP 
formula \eqref{SSP}
remains in $I_1$ for $h\in [0,2h_{\textrm{FE}}]$.

\begin{figure}
    \subfloat[SSP]{\includegraphics[width=3in]{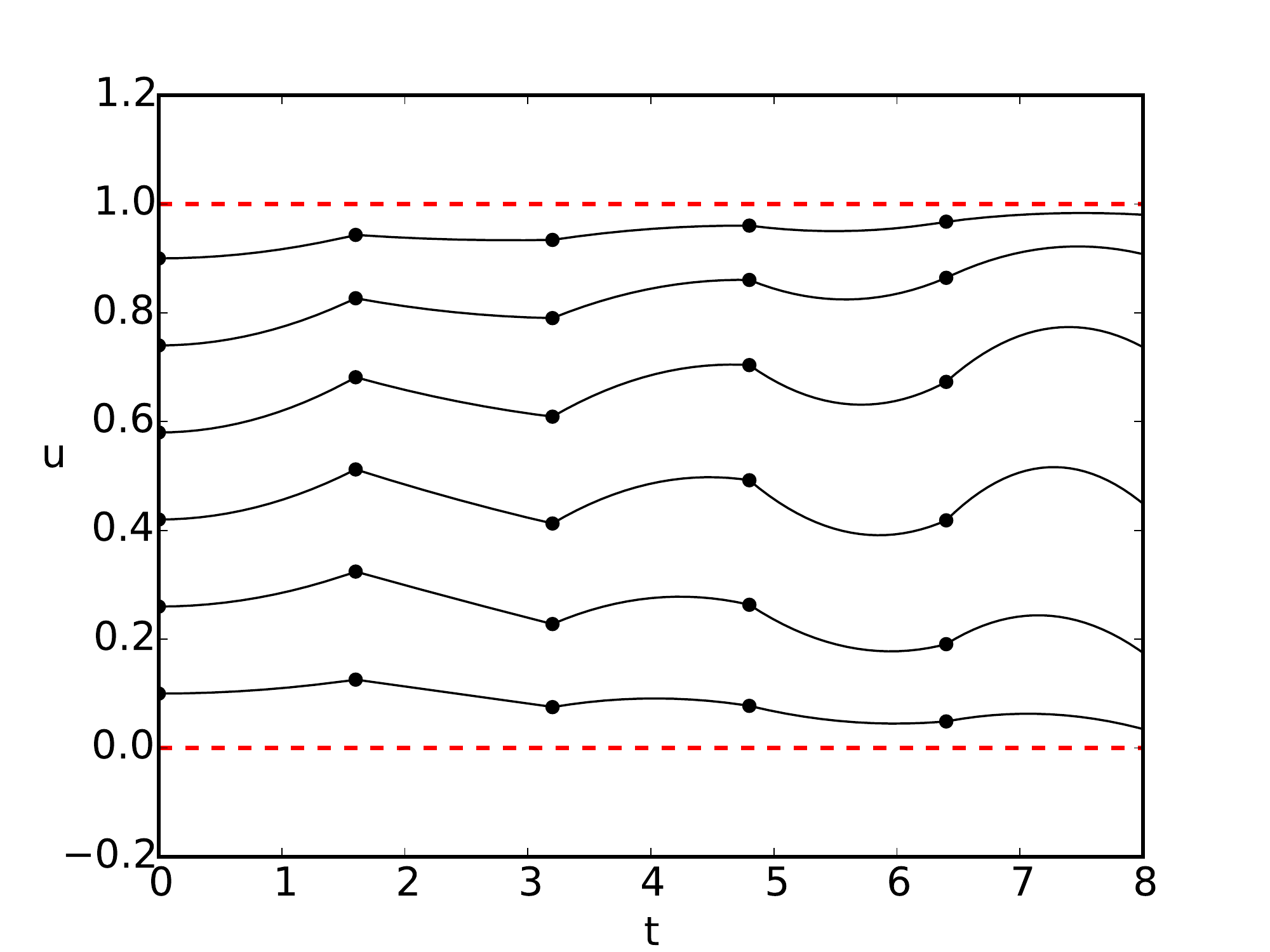}}
    \subfloat[non-SSP]{\includegraphics[width=3in]{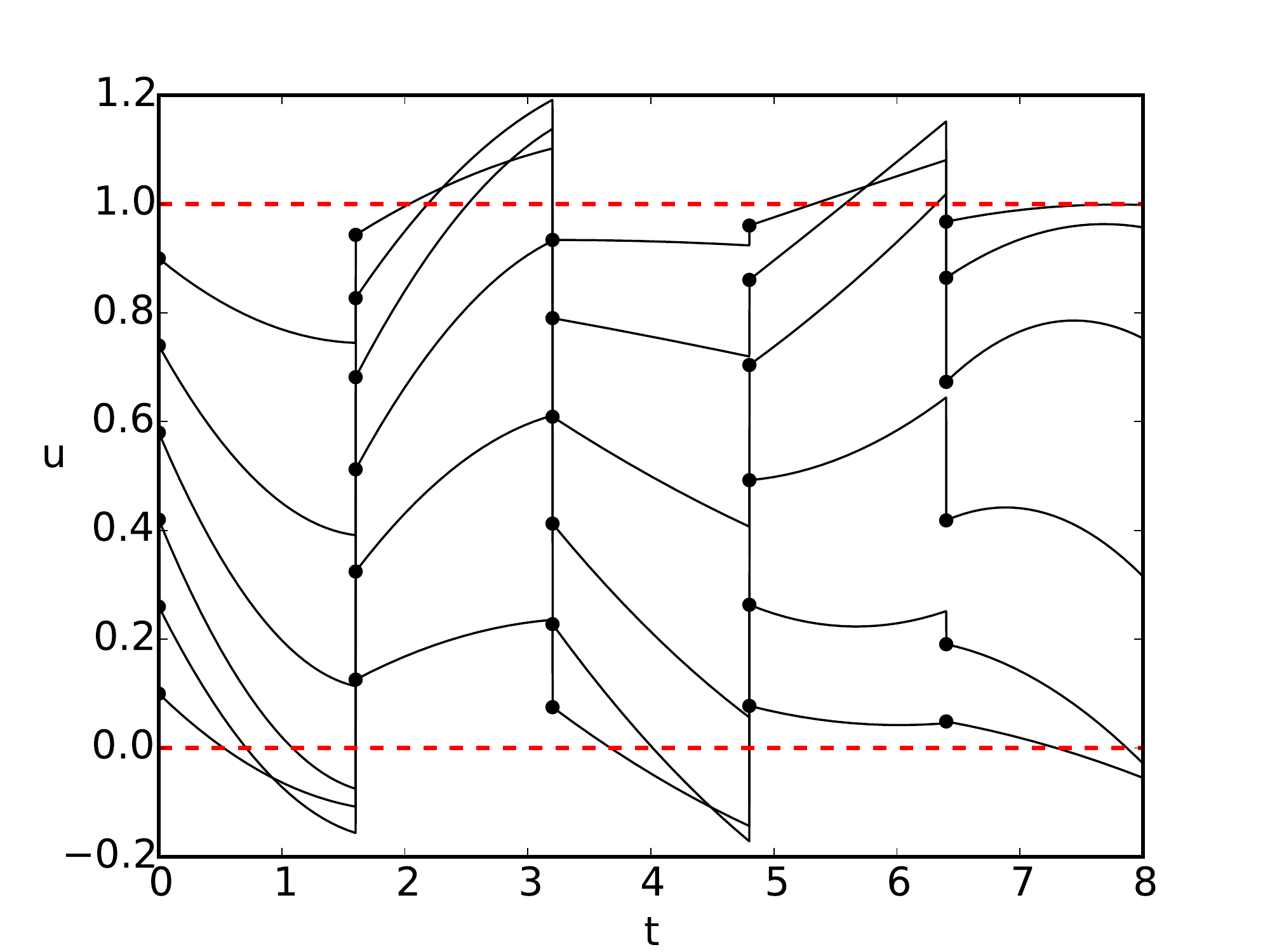}}
    \caption{Numerical solutions of \eqref{sinode} computed using 
    SSP dense output \eqref{SSP} (left figure) and non-SSP dense
        output \eqref{nonSSP} (right figure).\label{fig:sinode}}
 \end{figure}

\subsection{Outline}
In the remainder of the paper, we show how to obtain SSP dense
output formulae like \eqref{SSP} and investigate conditions under
which such formulae exist.  In Section
\ref{basic}, we recall algebraic properties of SSP methods and
order conditions for Runge--Kutta methods.  We also formulate
the algebraic conditions for SSP dense output formulae.  In
Section \ref{sec:restrict}, we show by considering the simpler
case of quadrature that third-order SSP dense output does not exist.
In Section \ref{formulae}, we give a first-order SSP dense output
for any SSP method and a second-order SSP dense output for
certain methods.
In the Appendix we list some practically useful and simple SSP dense output formulae.

\section{Background on strong stability preservation and dense output\label{basic}}
In this section we review some of the essential algebraic
properties of SSP RK methods and dense output that will be used later.
For a more general review of SSP theory, see \cite{SSPbook}.
For a more general review of dense output, see \cite{bellen1996some}.

\subsection{Runge--Kutta methods}
An $s$-stage RK method generates an approximate solution of the initial
value problem
\begin{align*}
	u'(t) & = f(u(t))
\end{align*}
via the iteration 
\begin{subequations}
\begin{align}
	y_i & = u_n + h  \sum_{j=1}^s a_{i,j} f(y_j) \\
    u_{n+1} & = u_n +  h \sum_{j=1}^s b_j f(y_j).\label{RKdef2eq}
\end{align}
\end{subequations}
For simplicity of notation we have assumed the problem is 
written in autonomous form. The method is defined by 
the coefficient arrays $A \in \Real^{s \times s},
b \in \Real^s$.  We will also refer to the vector $c$ of
abscissas, defined by $c_i := \sum_{j=1}^s a_{i,j}$.

The order conditions for a Runge--Kutta method up to order three are
\begin{subequations} \label{ocRK}
\begin{align}
    \sum_{j=1}^s b_j & =  1,  & (p=1)\  \label{ocRK1} \\
     \sum_{j=1}^s b_j c_j & =  \frac{1}{2}, &  (p=2)\  \label{ocRK2}  \\
    \sum_{j=1}^s  b_j c_j^2 & =  \frac{1}{3},  &  (p=3)\  \label{ocRK3} \\
    \sum_{j=1}^s b_j \left(\frac{c_j^2}{2} - \sum_{k=1}^s a_{j,k} c_k\right) & =  0,  &  (p=3). \label{(4d)}
\end{align}
\end{subequations}

Any method with more than one row of $A$ equal to zero is (trivially) reducible \cite{jeltsch2006reducibility}.
Furthermore, a Runge--Kutta method is invariant under permutation of its stages.
Therefore, without loss of generality we assume 
throughout this work that $A$ has at most one row identically
equal to zero, and if it has such a row then it is the first row:
\begin{align}\label{assumption-z}
    \text{For each } 2\le j\le s, \text{ row $j$ of $A$ is not identically zero.}
\end{align}

\subsection{Strong stability preserving Runge--Kutta methods}
Let $\ee$ denote the vector of length $s$ with all entries equal
to 1.
A method is said to be strong stability preserving if there
exists $r>0$ such that
\begin{subequations} \label{absmonRK}
\begin{align}
	A(I+rA)^{-1} & \ge 0 & rA(I+rA)^{-1}\ee \le 1 \label{Aineq} \\
    b^\top(I+rA)^{-1} & \ge 0 & rb^\top(I+rA)^{-1}\ee \le 1. \label{bineq}
\end{align}
\end{subequations}
The inequalities are meant componentwise.

The {\em SSP coefficient} of the method, denoted by $\sspcoeff(A,b)$,
can be defined as follows \cite{SSPbook}:
\begin{align*}
	\sspcoeff(A,b) := \sup 
    	\left\{r \ge 0 : (I + rA)^{-1} \text{ exists and \eqref{absmonRK} holds} \right\}
\end{align*}
if the set in the $\sup$ above is not empty, and $\sspcoeff(A,b) :=0$ otherwise.
The following lemma is well known; for a proof see \cite[p. 65]{SSPbook}.
\begin{lemma}\label{lemma1}
Let a RK method $(A,b)$ be given.  If the method has positive
SSP coefficient $\sspcoeff(A,b)>0$, then $a_{i,j}\ge 0$ and $b_j\ge 0$
for all $1\le i,j\le s$.
\end{lemma}

\subsection{Dense output for Runge--Kutta methods}
A dense output formula takes the form
\begin{align} \label{dense-output}
	u_{n+\theta} & = u_n + h \sum_{j=1}^s \bb_j(\theta)f(y_j);
\end{align}
the weights $b_j$ in \eqref{RKdef2eq} have been replaced by some real-valued
functions $\bb_j(\cdot)$.
Here $u_{n+\theta}$ is an approximation to the solution at time
$t_n + \theta(t_{n+1}-t_n)=t_n+\theta h$, and naturally $\theta\in [0,1]$. 
Hence, throughout the work, we assume that the weights $\bb_j$ are defined on the interval $[0,1]$;
in Section \ref{sec:restrict} we will impose some (minimal) smoothness assumptions on $\bb_j$, 
and in Section \ref{formulae} we will assume that the weights $\bb_j$ are polynomials (to develop 
the simplest possible dense output formulae).
In some cases it is advantageous to include
an additional term proportional to $f(u_{n+1})$ in \eqref{dense-output} (but this modification of
 \eqref{dense-output} will not be considered 
in the present work). 

The order conditions for the dense output formula (up to order three) are
\cite[Section~2.6]{hairer1993} 
\begin{subequations} \label{oc}
\begin{align}
    \sum_{j=1}^s \bb_j(\theta) & =  \theta, \quad\forall \theta\in[0,1] & (p=1)\  \label{oc1} \\
     \sum_{j=1}^s \bb_j(\theta) c_j & =  \frac{\theta^2}{2}, \quad\forall \theta\in[0,1] &  (p=2)\  \label{oc2}  \\
    \sum_{j=1}^s  \bb_j(\theta) c_j^2 & =  \frac{\theta^3}{3}, \quad\forall \theta\in[0,1] &  (p=3)\  \label{3rdorderA} \\
    \sum_{j=1}^s \sum_{k=1}^s \bb_j(\theta) a_{j,k} c_k & = \frac{\theta^3}{6},  \quad\forall \theta\in[0,1] &  (p=3). \label{(5d)}
\end{align}
\end{subequations}
Conditions \eqref{oc1}-\eqref{3rdorderA} (and \eqref{ocRK1}-\eqref{ocRK3})
are {\em quadrature conditions}; i.e.~they are necessary and sufficient for
third order convergence when the method is applied to a pure quadrature
problem $u'(t) = f(t)$.  

For a RK method of order $p$,
the convergence rate of the dense output values will be the same as
that of the RK step outputs as long as the 
dense output formula has order $p-1$ \cite{hairer1993}.

\subsubsection{Continuity at the endpoints}
It is natural to require $ \bb_j(0)=0$ and 

\begin{align} \label{rightendpointcont}
\bb_j(1) = b_j
\end{align}
for a dense output
formula, expressing continuity as $\theta\to 0^+$ and as $\theta \to 1^-$,
respectively, in \eqref{dense-output}.  
These assumptions are made in \cite{zennaro1986natural} and subsequent
works.  However, our interest is in
applying SSP methods to nonlinear hyperbolic PDEs, whose solutions are
not continuous.  Indeed, many modern spatial discretizations for 
hyperbolic PDEs are based on a representation that is discontinuous at every
node of the mesh.  Thus in our study we allow for extensions that are
not continuous at the endpoints.  As we will see, the SSP conditions
turn out to imply continuity as $\theta \to 0^+$, and the optimal
formulae we find all satisfy continuity also as $\theta \to 1^-$.

\subsection{Strong stability preserving dense output}
We are interested in dense output formulae
that also possess the SSP property.
To define the {\em SSP coefficient of the dense output formula}
we make use of the following inequalities, which are just \eqref{bineq}
with $b$ replaced by $\bb$:
\begin{align}
    \bb(\theta)^\top(I+rA)^{-1} & \ge 0 & r\bb(\theta)^\top(I+rA)^{-1}\ee & \le 1 
    & \text{ for } 0 \le \theta \le 1. \label{bthineq}
\end{align}
Here and in what follows,  $\bb:[0,1]\to \mathbb{R}^s$ denotes the function $(\bb_1,\ldots,\bb_s)$.
The SSP coefficient of the dense output formula, 
denoted by $\sspcoeff(A,\bb)$, is defined as follows:
\begin{align*}
	\sspcoeff(A,\bb) := \sup 
    	\left\{r\ge 0 : (I + rA)^{-1} \text{ exists and } 
        		\text{\eqref{Aineq} and \eqref{bthineq} hold}
                \text{ for all } 0\le \theta\le 1 \right\}
\end{align*}
if the set in the $\sup$ above is not empty, and $\sspcoeff(A,\bb):=0$ otherwise.
We also define the {\em SSP coefficient of the method with its dense
output formula} as
\begin{align*}
	\sspcoeff(A,b,\bb) & := 
    	\min\left( \sspcoeff(A,b), \sspcoeff(A,\bb)\right).
\end{align*}
This is the coefficient that matters in practice
since it dictates the step size that can be used while guaranteeing
strong stability of the solution at both the step points and dense output
points.

Nonnegativity of the dense output coefficients is necessary in order
that the SSP coefficient $\sspcoeff(A,b,\bb)$ be positive.
\begin{lemma} \label{lem:pos-coeff}
Let a RK method with dense output $(A,b,\bb)$ be given
such that $\sspcoeff(A,b,\bb)>0$.  
Then $a_{i,j}\ge 0$, $b_j\ge 0$, and $\bb_j(\theta) \ge 0$
for all $1\le i,j\le s$ and for all $0\le \theta \le 1$.
\end{lemma}
The proof of Lemma \ref{lem:pos-coeff} is similar to that of 
Lemma \ref{lemma1}.

The requirement of continuity as $\theta \to 0^+$ turns out to be necessary for
SSP dense output.
\begin{lemma}\label{lemma3}
	Let a RK method with dense output of order at least one be given 
    with coefficients $(A,b,\bb)$.  
    If $\sspcoeff(A,b,\bb)>0$, then 
    \begin{align} \label{zerocond}
      \bb_j(0)=0 \quad(1\le j\le s).
    \end{align}
\end{lemma}
\begin{proof}
	By Lemma \ref{lem:pos-coeff} we have $\bb_j(\theta)\ge 0$, and condition
    \eqref{oc1} implies $\sum_j \bb_j(0) = 0$.
\end{proof}

Finally, it is clear from the definitions that
\begin{align} \label{ineqbetweenCs}
\eqref{rightendpointcont} \implies \sspcoeff(A,\bb)\le \sspcoeff(A,b).
\end{align}

\section{Restrictions and non-existence results\label{sec:restrict}}
It turns out that certain restrictions on SSP dense output formulae
appear already when one considers only quadrature problems.  The case
of quadrature is considered in Section \ref{sec:quadrature}, and
general results for Runge--Kutta methods are deduced in Section \ref{barrier}.

In this section we are going to impose some smoothness assumptions on the weight functions 
$\bb_j$ defined on $[0,1]$. To formulate these, we will use the notation 
\[{\mathcal{D}}^{k}(0+)\quad (k\in \mathbb{N}^+)\]
to denote
the set of real functions that are $k$ times right-differentiable at the origin; the symbols $\bb_{j}'(0)$
and $\bb_{j}''(0)$ will denote right-derivatives.

\subsection{A result on non-negative quadrature rules\label{sec:quadrature}}
When we apply a Runge--Kutta method to a problem in which $f$ is independent of $u$,
\begin{align}
    u'(t) & = f(t), & u(t_0) = u_0,
\end{align}
 the dense output reduces to a quadrature rule
\begin{align} \label{quadrule}
    u(t_0 + \theta) \approx u_0 + \sum_{j=1}^s \bb_j(\theta) f(t_0 + \theta c_j).
\end{align}
We assume, without loss of generality, that
the abscissas $c_j$ are distinct (any method \eqref{quadrule} with repeated
abscissas can be rewritten as a method with distinct abscissas).
We show that if the abscissas $c_j$ and weights $\bb_j$ are non-negative,
the accuracy of \eqref{quadrule} is limited.

\begin{theorem}\label{thmquadrule}
    Let a quadrature rule \eqref{quadrule} be given with distinct abscissas
    $c_j\in\mathbb{R}$, and with weight functions $\bb_j\in {\mathcal{D}}^{1}(0+)$ 
that are non-negative in a
    right neighborhood of zero, i.e.
    \begin{align}\label{rightneighborhoodcond}
        \bb_j(\theta) & \ge 0  & \forall j \text{ and }\  \forall \theta & \in [0,\epsilon) \text{ with some } \epsilon>0.
    \end{align}
    If \eqref{quadrule} is exact for quadratic polynomials, i.e.~if \eqref{oc1}-\eqref{oc2} hold, 
    then at least one abscissa is non-positive.
    If  $\bb_j\in {\mathcal{D}}^{2}(0+)$, and \eqref{quadrule} is exact for cubic polynomials, i.e.~if \eqref{oc1}-\eqref{3rdorderA} hold, then at least one abscissa is
    negative.
\end{theorem}
\begin{proof}
    We suppose in the proof that $c_j\ge 0$ ($1\le j\le s$) (\textit{otherwise the proof is complete}).
    Formula \eqref{oc1} and \eqref{rightneighborhoodcond} at $\theta=0$ imply 
    \begin{equation}\label{thm6prooflabel1}
    \bb_{j}(0)=0 \quad (1\le j\le s).  
     \end{equation}
    But then \eqref{rightneighborhoodcond} also implies $\bb'_{j}(0)\ge 0$.
    Differentiation of \eqref{oc2} at $0$ yields $\sum_{j=1}^s \bb'_{j}(0)c_j = 0$, so from 
    the non-negativity we get 
\begin{equation}\label{thm6prooflabel}
    \bb_{j}'(0) = 0 \text{ or } c_j=0 \text{ for each } j.
\end{equation}
    Differentiation of \eqref{oc1} at $0$ shows that $\bb_{j_0}'(0) \ne 0$ for some $j_0$, hence $c_{j_0}=0$.

    \textit{To prove the second statement,} we can thus assume that exactly one abscissa is zero
    (since they are distinct).
    Let the abscissas be ordered so that $c_1=0$; then $c_j>0$ for $2\le j \le s$.
    The derivative of \eqref{oc1} at $\theta = 0$ and \eqref{thm6prooflabel} now imply 
\begin{equation}\label{repetition}
    \bb_{1}'(0) = 1 \text{ and } \bb_j'(0) = 0 \text{ for } 2 \le j \le s.
\end{equation}    
     So from \eqref{rightneighborhoodcond} and 
    \eqref{thm6prooflabel1} we obtain that
    $\bb_{j}''(0)\ge 0$ for $2 \le j\le s$.  Since \eqref{3rdorderA} says
    \[\sum_{j=1}^s \bb_{j}''(0) c_j^2 = \sum_{j=2}^s \bb_{j}''(0) c_j^2= 0,\]
    we also have $\bb_{j}''(0)=0$ for all $2 \le j\le s$.  But this is incompatible with the second derivative of \eqref{oc2} at $\theta=0$.
\end{proof}

\subsection{Restrictions on SSP dense ouput for Runge--Kutta methods\label{barrier}}
For a given SSP RK method $(A,b)$ of $s$ stages and order $p$, 
we seek to find functions 
$\bb_j$ ($1 \le j \le s$) that satisfy the order 
conditions \eqref{oc} up to at least order $p-1$ and the 
SSP conditions \eqref{bthineq}.
The next theorem restricts the form of possible formulae
and methods with positive SSP coefficient.
\begin{theorem} \label{thm2}
Let an SSP Runge--Kutta method be given with 
coefficients $(A,b)$ where $A$ satisfies \eqref{assumption-z},
and let $(A,\bb)$ be a 
dense output formula of order at least two with 
weight functions $\bb_j\in {\mathcal{D}}^{1}(0+)$.
Suppose that $\sspcoeff(A,b,\bb)>0$.
Then the first row of $A$ is identically zero and
\begin{subequations} \label{2coeffcond}
  \begin{align}
  \bb'_{1}(0) & = 1 \\ 
   \bb'_{j}(0) & = 0 & \text{ for all } 2\le j  \le s.
  \end{align}
\end{subequations}
\end{theorem}
\begin{proof}
    Consider a method that satisfies the assumptions in the theorem.
    By Lemma \ref{lem:pos-coeff} we have $\bb_j\ge 0$ and $c_j\ge0$ ($1\le j\le s$).
        The method is 2nd-order accurate for quadratures, 
        so  we can reread the proof of Theorem
    \ref{thmquadrule} up to \eqref{repetition} but omitting  the parts written \textit{in italics}.  
    Note also that by assumption \eqref{assumption-z} we have $c_1 = 0$ (and
    $c_j \ne 0$ for $2\le j \le s$)  in the second part of that proof.
\end{proof}

The next result indicates that no dense output formula of order three or higher
exists. 
\begin{theorem}\label{thmnegresultorder3}
Let an SSP Runge--Kutta method $(A,b)$ be given such that $A$ satisfies \eqref{assumption-z},
along with a 
dense output formula $\bb\in {\mathcal{D}}^{2}(0+)$
such that 
$\sspcoeff(A,b,\bb)>0$.  Then the order of accuracy of the
dense output formula is at most two.
\end{theorem}
\begin{proof}
This follows from Lemma \ref{lem:pos-coeff} and Theorem \ref{thmquadrule},
which imply that an SSP dense output formula cannot be third order accurate
even for quadratures.
\end{proof}

\section{Construction of 1st- and 2nd-order formulae\label{formulae}}
In this section we develop dense output formulae 
for the optimal SSP Runge--Kutta methods of orders 1 and 2.  As always,
we assume condition \eqref{assumption-z}.
It turns out that dense output formulae of order one or two can
be obtained in a simple, general way for many SSP methods.

From now on it is natural to assume that each weight $\bb_j$ is a polynomial whose 
degree, $D$,  is equal to or greater than the desired order of accuracy 
of the dense output formula, $p-1\le D$.
That is, throughout Section \ref{formulae}, we represent $\bb_j$ as
\begin{align}\label{bthetarepresentation}
\bb_j(\theta) & = \sum_{k=0}^{D} \bb_{j,k} \theta^k
\end{align}
with some coefficients $\bb_{j,k}\in \mathbb{R}$ ($1\le j\le s$, $0\le k\le D$).

\subsection{First-order dense output}

\begin{theorem}
Let a Runge--Kutta method of order at least one be given with 
coefficients $(A,b)$ and SSP coefficient
$\sspcoeff(A,b)>0$.  Then a first-order dense output with the same SSP 
coefficient $\sspcoeff(A,\bb)=\sspcoeff(A,b)$ is obtained by taking $\bb_j(\theta) := b_j \theta$ ($1\le j\le s$);
i.e.
\begin{align} \label{1storder}
    u_{n+\theta} & = u_n + h \,\theta \sum_{j=1}^s b_j f(y_j).
\end{align}
\end{theorem}
\begin{proof}
It suffices to check that condition \eqref{oc1} is satisfied
and conditions \eqref{bthineq} hold with $r=\sspcoeff(A,b)$.
By taking into account $\bb(\theta) =  b \theta$ ($\theta\in [0,1]$), condition \eqref{oc1} follows from \eqref{ocRK1}, 
and conditions \eqref{bthineq} follow from the fact that method
$(A,b)$ satisfies \eqref{bineq} with $r=\sspcoeff(A,b)$.
\end{proof}

\subsection{Second-order dense output}\label{sec:denseoutputorder2}
Next we turn our attention to second-order dense output.
According to Theorem \ref{thm2}, and applying assumption \eqref{assumption-z} 
we have that the first row of $A$ is zero.
Motivated by \eqref{zerocond}, \eqref{2coeffcond}, \eqref{rightendpointcont} and \eqref{oc1}, 
and taking $D=2$ in \eqref{bthetarepresentation}, 
we see that these conditions yield a unique set of quadratic polynomials $\bb_j$:
\begin{subequations} \label{2ndorder_coeffs}
\begin{align}
	\bb_1(\theta) & = \theta - (1-b_1)\theta^2 \\
    \bb_j(\theta) & = \theta^2 b_j & (2\le j\le s).
\end{align}
\end{subequations}
For many SSP RK methods $(A,b)$, this choice of $\bb$ gives a second-order SSP dense output formula with $\sspcoeff(A,\bb) = \sspcoeff(A,b)$ (cf.~\eqref{ineqbetweenCs}).

\begin{theorem} \label{DO2thm}
Let a Runge--Kutta method of order at least two be given 
with coefficients $(A,b)$ and SSP coefficient $\sspcoeff(A,b)>0$.
Suppose that the first row of $A$ is zero and let $\bb$ be defined by \eqref{2ndorder_coeffs}. Then the following statements are true.
\begin{enumerate}
\item $(A,\bb)$ is an SSP dense output formula
of order at least two. 
\item If 
\begin{align} \label{xineq}
  b^\top(I + \sspcoeff(A,b) A)^{-1} \ee \le 1 - \frac{\sspcoeff(A,b)}{4},
\end{align}
then $\sspcoeff(A,\bb) = \sspcoeff(A,b)$. 
\item If $\sspcoeff(A,b)\le 2$,
then $\sspcoeff(A,\bb)= \sspcoeff(A,b)$.
\end{enumerate}
\end{theorem}
\begin{proof}
Since formulae \eqref{2ndorder_coeffs} satisfy \eqref{rightendpointcont}, we have 
 $\sspcoeff(A,\bb)\le \sspcoeff(A,b)$ due to 
\eqref{ineqbetweenCs}.
So it suffices to check that the order conditions \eqref{oc1}-\eqref{oc2}
are satisfied and that the SSP conditions \eqref{bthineq} hold
with $r=\sspcoeff(A,b)$.  

The order conditions are easily checked, since $\sum_{j=1}^s \bb_j(\theta)  =  
\theta + \theta^2\left(-1+\sum_{j=1}^s b_j\right)=\theta$ due to \eqref{ocRK1}.
Similarly, since $c_1=0$ because the first row of $A$ is zero, we have $\sum_{j=1}^s \bb_j(\theta) c_j  = \sum_{j=2}^s \bb_j(\theta) c_j =\theta^2 \sum_{j=2}^s b_j c_j=\theta^2 \sum_{j=1}^s b_j c_j=\frac{\theta^2}{2}$ due to \eqref{ocRK2}.

To verify the first SSP condition in \eqref{bthineq}, we let $M := (I+\sspcoeff(A,b) A)^{-1}$
and $\ee_1 := (1,0,\dots,0)^\top$.
Since the first row of $A$ is zero, the first row of $M$ is $\ee_1^\top$, hence
$\ee_1^\top M = \ee_1^\top$.
Observe that 
$\bb(\theta) = \theta^2 b + (\theta-\theta^2)\ee_1$, so for $0 \le \theta \le 1$ we have
$\bb(\theta)^\top M  = \theta^2 b^\top M + (\theta-\theta^2)\ee_1^\top M=\theta^2 b^\top M + (\theta-\theta^2)\ee_1^\top \ge \theta^2 b^\top M \ge 0$, because $b$ satisfies \eqref{bineq}.

As for the second SSP condition in \eqref{bthineq}, we set $\gamma:=b^\top M \ee$ and $\varphi(\theta):=\sspcoeff(A,b)(\gamma-1)\theta^2 + \sspcoeff(A,b)\theta-1$ to rewrite it as
\[
	\sspcoeff(A,b) \bb(\theta)^\top M \ee  = \sspcoeff(A,b) \theta^2 \gamma + \sspcoeff(A,b) (\theta-\theta^2)\ee_1^\top M \ee = \]
\[
 \sspcoeff(A,b) \theta^2 \gamma + \sspcoeff(A,b) (\theta-\theta^2)=
				 \varphi(\theta)+1.
\]
Hence 
$\sspcoeff(A,b) \bb(\theta)^\top M \ee\le 1$ is equivalent to 
\begin{equation}\label{phiineq}
\varphi(\theta)\le 0 \quad \forall\theta\in [0,1].
\end{equation} 
In what follows we will use 
\begin{equation}\label{auxineq}
\sspcoeff(A,b)>0, \quad \gamma\ge 0\quad\text{and}\quad \sspcoeff(A,b) \gamma\le 1 
\end{equation}
to show \eqref{phiineq}; the last two inequalities of \eqref{auxineq} follow from \eqref{bineq} again.\\ 
\indent $\bullet$ For $\gamma=1$, we have $\sspcoeff(A,b)\le 1$ and $\varphi(\theta)=
\sspcoeff(A,b)\theta-1$, hence \eqref{phiineq} holds and the proof is complete.\\
\indent $\bullet$ For $\gamma>1$, the strict global minimum of the parabola $\varphi$ is attained at $\theta^*=\frac{1}{2(1-\gamma)}<0$. Hence \eqref{phiineq} is equivalent to $\varphi(1)\le 0$. But 
$\varphi(1)=\sspcoeff(A,b)\gamma-1\le 0$ due to \eqref{auxineq}.\\ 
\indent $\bullet$ For $\gamma\in [0,1)$, the strict global maximum of the parabola $\varphi$ is attained at $\theta^*=\frac{1}{2(1-\gamma)}>0$. For $\gamma\in \left(\frac{1}{2},1\right)$, we have $\theta^*>1$, so
\eqref{phiineq} is equivalent to $\varphi(1)\le 0$, which we have already seen to hold.
Finally, for $\gamma\in \left[0,\frac{1}{2}\right]$, we see that $\theta^*\in(0,1]$, so \eqref{phiineq} is equivalent to $\varphi(\theta^*)\le 0$. Here $\varphi(\theta^*)=\frac{\sspcoeff(A,b)+4\gamma-4}{4 (1-\gamma )}$, and
\eqref{xineq} or $\sspcoeff(A,b)\le 2$ implies $\varphi(\theta^*)\le 0$. 
\end{proof}

\subsection{Implicit SSP methods}
First-order SSP dense output for any implicit SSP method can be constructed
according to \eqref{1storder}.  However, since all known optimal implicit SSP methods are
diagonally implicit, Theorem \ref{thm2} implies that they have no second-order
SSP dense output with polynomial weights.  It is possible that there exist implicit methods with the
first row of $A$ equal to zero such that $\sspcoeff(A,b)>0$ and some
$\bb$ such that $\sspcoeff(A,\bb)>0$.  However, since such methods will have
SSP coefficient even smaller than those of the optimal implicit SSP methods,
we have not conducted a search for them.

\subsection{Explicit SSP methods}
The theorem in Section \ref{sec:denseoutputorder2} does not settle the question of 
existence of dense output formulae of order two for particular RK methods
of interest.  
In this section we develop such formulae for some 
optimal explicit SSP RK methods $(A,b)$.

One can directly verify that the following optimal explicit SSP methods satisfy condition \eqref{xineq} 
in Theorem \ref{DO2thm}, and
so \eqref{2ndorder_coeffs} provides a 2nd-order SSP dense output formula $\bb$ 
with $\sspcoeff(A,\bb)=\sspcoeff(A,b)>0$:
\begin{itemize}
    \item second-order methods with two \cite{shu1989efficient}, three, or four \cite{spiteri2002} stages;
    \item third-order methods with three \cite{shu1989efficient} or four \cite{spiteri2002} stages;
    \item the fourth-order method with five stages \cite{spiteri2002}.
\end{itemize}
Coefficients for all of the methods above can be found, for instance in
\cite[Chapter 6]{SSPbook}.
Computations show that explicit optimal methods with more stages than those listed above
do not satisfy \eqref{xineq} and we have not found a way to construct
dense output formulae with SSP coefficient as large as that of the method
itself for any of them.

As an illustration, we show in the following subsection that
there exists a 2nd-order $s$-stage dense output formula 
with $\sspcoeff(A,\bb)\ge \sspcoeff(A,b)$ where the weights $\bb_{j}$ are quadratic 
polynomials in $\theta$ if and only if $s\in\{2,3,4\}$.

\subsection{Dense output for optimal explicit second-order SSP RK methods}\label{subsection41}
A family of optimal 2nd-order SSP Runge--Kutta methods with $s$ stages
(for $s\ge 2$) was introduced in \cite{spiteri2002}.  We have already
seen above that \eqref{2ndorder_coeffs} gives an SSP dense output with
$\sspcoeff(A,\bb) = \sspcoeff(A,b)$ for the members of this
family with $s=2,3,4$.  In this section we show that no quadratic
dense output with $\sspcoeff(A,\bb) = \sspcoeff(A,b)$ exists for
the members of this family with $s\ge 5$.

We will use the following properties 
of an optimal 2nd-order SSP RK method $(A,b)$ with $s$ stages \cite{SSPbook}: 
it satisfies \eqref{assumption-z}; it has $b_j=1/s$ and $c_j = (j-1)/(s-1)$ for
$1\le j \le s$; $\sspcoeff(A,b)=s-1$; and it satisfies  
\begin{align*}
(I+\sspcoeff(A,b) A)^{-1} & = \begin{pmatrix} 
1  &   &   &    \\
-1 & 1 &   &   \\
   & \ddots & \ddots &  \\
   &   & -1 & 1 
\end{pmatrix}
\end{align*}
(with zeros elsewhere). 

A necessary condition for $\sspcoeff(A,\bb)\ge \sspcoeff(A,b)$ is that the SSP conditions \eqref{bthineq} with $r:=\sspcoeff(A,b)$ are fulfilled.
Thus conditions \eqref{bthineq} lead to the inequalities
\begin{align} \label{orig2ineq}
0 \le \bb_{s}(\theta) \le \bb_{s-1}(\theta) \le \ldots \le \bb_{2}(\theta) \le \bb_1(\theta) \le \frac{1}{\sspcoeff(A,b)},\quad \forall\theta\in [0,1].
\end{align}
On the one hand, by applying Lemma \ref{lemma3}, Theorem \ref{thm2}, condition \eqref{oc1} and
the representation \eqref{bthetarepresentation}, we get that 
\begin{subequations}\label{label18ab}
\begin{align}
\bb_j(\theta) & = \bb_{j,2}\, \theta^2  \quad \quad (2 \le j \le s) \\
\bb_1(\theta)  & = \theta - \theta^2 \sum_{j=2}^s \bb_{j,2}
\end{align}
\end{subequations}
should hold with some constants $\bb_{j,2}$.
Hence \eqref{orig2ineq} evaluated at $\theta=1$ leads to
the necessary conditions
\begin{subequations}
\begin{align}
0 \le \bb_{s,2} \le \bb_{s-1,2} \le \ldots \le \bb_{2,2} \le 1 - \sum_{j=2}^s \bb_{j,2} \label{2ineqa}, \\
1 - \sum_{j=2}^s \bb_{j,2} \le \frac{1}{s-1}. \label{2ineqb}
\end{align}
\end{subequations}
On the other hand, the second-order accuracy condition \eqref{oc2} evaluated at $\theta=1$ now reads
\begin{align} \label{2cond}
	\sum_{j=2}^s \frac{j-1}{s-1} \bb_{j,2} = \frac{1}{2}.
\end{align}
One way to satisfy \eqref{2cond} and \eqref{2ineqa} is to take 
\begin{align}
\bb_{j,2} & = \frac{1}{s}\quad (2\le j\le s).
\end{align}
It turns out that this solution is unique.
\begin{lemma}\label{uniquenesslemma}
Suppose that the real numbers $\bb_{2,2}, \dots, \bb_{s,2}$ satisfy \eqref{2cond}
and \eqref{2ineqa}.  Then $\bb_{j,2} = 1/s$ for $2\le j\le s$.
\end{lemma}
\begin{proof}
Set $ \delta_j:=\bb_{j,2} - \frac{1}{s} $. Then from \eqref{2cond} we get
\begin{align} \label{epseq}
	\sum_{j=2}^s (j-1) \delta_j = 0,
\end{align}
and \eqref{2ineqa} becomes
\begin{align} \label{epsineq}
	-1/s \le \delta_s \le \ldots \le \delta_{2} \le - \sum_{j=2}^s \delta_j.
\end{align}
We will show that \eqref{epseq} and \eqref{epsineq} imply that $\delta_j=0$ for each $j$.
Suppose by way of contradiction that $\delta_j\ne 0$ for some $j$.  Then \eqref{epseq}
implies that $\delta_j>0$ for some $j$, so by \eqref{epsineq} we have $\delta_2>0$.
Hence \eqref{epsineq} implies that 
$\sum_{j=2}^s \delta_j < 0$.
This in turn implies recursively that each of the partial sums $\sum_{j=i}^s \delta_j$
(for $i\ge 2$) is negative.  But then
\[
	\sum_{j=2}^s (j-1) \delta_j  = \sum_{i=2}^s \sum_{j=i}^s \delta_j < 0,
\]
which contradicts \eqref{epseq}.
\end{proof}

We can now give a complete characterization of quadratic 2nd-order SSP dense output
formulae for optimal 2nd-order SSP Runge--Kutta methods, satisfying $\sspcoeff(A,\bb)\ge \sspcoeff(A,b)$.

\begin{theorem} \label{2ndorderthm}
	Consider the optimal 2nd-order SSP RK method with $s$ stages,
    having SSP coefficient $\sspcoeff(A,b)=s-1$.
    For $2 \le s \le 4$,
    the method possesses a unique 2nd-order accurate SSP 
    dense output formula with quadratic polynomials $\bb_j$ and SSP coefficient $\sspcoeff(A,\bb)=s-1$:
\begin{subequations}\label{thm5formulae}
    \begin{align}
    	\bb_1(\theta) & := \theta - \frac{s-1}{s} \theta^2 \\
        \bb_j(\theta) & := \frac{1}{s} \theta^2 & (2 \le j \le s).
    \end{align}
\end{subequations}
    For $s\ge 5$, there exists no 2nd-order SSP dense output formula
    quadratic in $\theta$ with $\sspcoeff(A,\bb)\ge s-1$.
\end{theorem}
\begin{proof} Formulae \eqref{label18ab} and 
Lemma \ref{uniquenesslemma} show that any candidate dense output formula must have 
the form given by \eqref{thm5formulae}. But now $b_j=1/s$ ($1\le j\le s$), hence formulae
\eqref{thm5formulae} are identical to formulae \eqref{2ndorder_coeffs}, and so Theorem \ref{DO2thm}
is applicable.

For $s\in\{2,3\}$, we have $\sspcoeff(A,b)=s-1\le 2$, so Statement 3 of Theorem \ref{DO2thm}
yields $\sspcoeff(A,\bb)= s-1$.

For $s=4$, we have equality in \eqref{xineq} because both sides are equal to $1/4$, so we again get $\sspcoeff(A,\bb)= s-1$.

For $s\ge 5$ however, the inequality $\bb_1(\theta) \le \frac{1}{\sspcoeff(A,b)}$ in 
the necessary condition \eqref{orig2ineq}
 is violated for 
$\theta:=\frac{s}{2(s-1)}\in (0,1]$, so there is no corresponding 2nd-order SSP dense output formula
    with $D=2$ and $\sspcoeff(A,\bb)\ge \sspcoeff(A,b)$.
\end{proof}

\begin{remark}
For optimal 2nd-order explicit SSP RK methods $(A,b)$ with $s\ge 5$ stages,
Theorem \ref{2ndorderthm} does not preclude the existence of 2nd-order dense output
formulae that consist of $\bb_j$ polynomials of 
 degree $D\ge 3$ in \eqref{bthetarepresentation} and have $\sspcoeff(A,\bb)\ge \sspcoeff(A,b)>0$.  But initial searches have failed to yield any such formulae. 
\end{remark}

\begin{remark}
The uniqueness in Theorem \ref{2ndorderthm} ``almost'' follows from the uniqueness stated in the sentence preceding formula \eqref{2ndorder_coeffs}.
However,  when deriving \eqref{2ndorder_coeffs}, the right endpoint condition \eqref{rightendpointcont}
has been taken into account---whereas when deriving Theorem \ref{2ndorderthm}, one does not use \eqref{rightendpointcont}.
\end{remark}

\section{Conclusion}
We have proved the non-existence of SSP dense output of order three or higher;
we have shown that this follows from a more fundamental result on the abscissas
of continuously-valued quadrature rules with positive weights.  On the other
hand, for some of the most common SSP methods it is straightforward to
construct dense output of second order using \eqref{2ndorder_coeffs}.

For methods that do not satisfy condition 2 or 3 of Theorem \ref{DO2thm}, it
may still be possible to develop 2nd-order dense output with 
$\sspcoeff(A,\bb) \ge \sspcoeff(A,b)$.

Theorem \ref{thmquadrule} naturally suggests the inclusion
of negative abscissas in the dense output formula.  This can
be done most simply by using previous step values $u_{n-j}$,
which is the subject of current research.


\vspace{1in}
{\bf Acknowledgment.}  We thank an anonymous referee for several suggestions
that improved the presentation of this work.

\appendix

\section{Some complete methods in Shu--Osher form\label{sec:appendix}}
We have worked with methods in Butcher form because of the simplicity
of the order conditions in that form.
In this appendix we write out some of the methods with dense output in
Shu--Osher form, since this is usually the most convenient form for
implementation.  For the sake of brevity we write the methods in
autonomous form and let $\sspcoeff=\sspcoeff(A,b,\bb)$.

The dense output in Shu--Osher form is
\begin{align*}
    u_{n+\theta} & = \mu u_n + \sum_{j=1}^s \betab_j(\theta) \left(y_j + \frac{h}{\sspcoeff}f(y_j)\right)
\end{align*}
where
\begin{align*}
    \betab^\top & := \bb^\top\left(I+\sspcoeff A\right)^{-1} \\
    \mu & := 1 - \sspcoeff \bb^\top \left(I + \sspcoeff A\right)^{-1}.
\end{align*}

We denote each method by SSP$(s,p,\overline{p})$, where $s$ is the number of
stages, $p$ is the order of the method $(A, b)$, and $\overline{p}$ is the order of the dense output.

\subsection{SSP(2,2,2)}
This method has $\sspcoeff(A,b,\bb) = 1$.
\begin{align*}
    y_1 & = u_n \\
    y_2 & = y_1 + h f(y_1) \\
    u_{n+1} & = \frac{1}{2} u_n + \frac{1}{2} \left(y_2 + hf(y_2)\right) \\
    u_{n+\theta} & = \left(1 - \theta + \frac{1}{2}\theta^2\right)u_n + 
        (\theta-\theta^2)\left(y_1+hf(y_1)\right) +
        \frac{1}{2}\theta^2\left(y_2 + hf(y_2)\right).
\end{align*}

\subsection{SSP(3,2,2)}
This method has $\sspcoeff(A,b,\bb) = 2$.
\begin{align*}
    y_1 & = u_n \\
    y_2 & = y_1 + \frac{h}{2} f(y_1) \\
    y_3 & = y_2 + \frac{h}{2} f(y_2) \\
    u_{n+1} & = \frac{1}{3} u_n + \frac{2}{3} \left(y_3 + \frac{h}{2} f(y_3)\right) \\
    u_{n+\theta} & = \left(1-2\theta + \frac{4}{3}\theta^2\right)u_n + 
        2(\theta-\theta^2)\left(y_1+\frac{h}{2}f(y_1)\right) +
        \frac{2}{3}\theta^2\left(y_3 + \frac{h}{2}f(y_3)\right).
\end{align*}

\subsection{SSP(3,3,2)}
This method has $\sspcoeff(A,b,\bb) = 1$.
\begin{align*}
    y_1 & = u_n \\
    y_2 & = y_1 + h f(y_1) \\
    y_3 & = \frac{3}{4} u_n + \frac{1}{4} \left(y_2 + h f(y_2)\right) \\
    u_{n+1} & = \frac{1}{3} u_n + \frac{2}{3} \left(y_3 + h f(y_3)\right) \\
    u_{n+\theta} & = \left(1 - \theta + \frac{1}{3}\theta^2\right)u_n + 
        (\theta-\theta^2)\left(y_1+hf(y_1)\right) +
        \frac{2}{3}\theta^2\left(y_3 + hf(y_3)\right).
\end{align*}

\bibliography{bib}

\begin{thebibliography}{10}

\bibitem{bellen1997contractivity}
A.~Bellen.
\newblock Contractivity of continuous {R}unge--{K}utta methods for delay
  differential equations.
\newblock {\em Applied numerical mathematics}, 24(2):219--232, 1997.

\bibitem{bellen1996some}
A.~Bellen and R.~Vermiglio.
\newblock Some applications of continuous {R}unge--{K}utta methods.
\newblock {\em Applied numerical mathematics}, 22(1):63--80, 1996.

\bibitem{Enright:1986:IRF:7921.7923}
W.~H. Enright, K.~R. Jackson, S.~P. N{\o}rsett, and P.~G. Thomsen.
\newblock Interpolants for {R}unge--{K}utta formulas.
\newblock {\em ACM Trans. Math. Softw.}, 12(3):193--218, September 1986.

\bibitem{2009_ssp_review}
Sigal Gottlieb, David~I. Ketcheson, and Chi-Wang Shu.
\newblock {High Order Strong Stability Preserving Time Discretizations}.
\newblock {\em Journal of Scientific Computing}, 38(3):251--289, 2009.

\bibitem{SSPbook}
Sigal Gottlieb, David~I. Ketcheson, and Chi-Wang Shu.
\newblock {\em Strong Stability Preserving {R}unge--{K}utta And Multistep Time
  Discretizations}.
\newblock World Scientific, January 2011.

\bibitem{hairer1993}
Ernst Hairer, Syvert~P. N{\o}rsett, and Gerhard Wanner.
\newblock {\em {Solving ordinary differential equations I: Nonstiff Problems}}.
\newblock Springer Series in Computational Mathematics. Springer, Berlin,
  second edition, 1993.

\bibitem{jeltsch2006reducibility}
Rolf Jeltsch.
\newblock Reducibility and contractivity of {R}unge--{K}utta methods revisited.
\newblock {\em BIT Numerical Mathematics}, 46(3):567--587, 2006.

\bibitem{macdonald2008numerical}
Colin~B. Macdonald, Sigal Gottlieb, and Steven~J. Ruuth.
\newblock A numerical study of diagonally split {R}unge--{K}utta methods for
  {PDE}s with discontinuities.
\newblock {\em Journal of Scientific Computing}, 36(1):89--112, 2008.

\bibitem{shu1989efficient}
Chi-Wang Shu and Stanley Osher.
\newblock Efficient implementation of essentially non-oscillatory
  shock-capturing schemes, ii.
\newblock {\em Journal of Computational Physics}, 83(1):32--78, 1989.

\bibitem{spiteri2002}
Raymond~J. Spiteri and Steven~J. Ruuth.
\newblock {A New Class of Optimal High-Order Strong-Stability-Preserving Time
  Discretization Methods}.
\newblock {\em SIAM Journal on Numerical Analysis}, 40:469--491, 2002.

\bibitem{torelli1991sufficient}
Lucio Torelli.
\newblock A sufficient condition for {GPN}-stability for delay differential
  equations.
\newblock {\em Numerische Mathematik}, 59(1):311--320, 1991.

\bibitem{zennaro1986natural}
Marino Zennaro.
\newblock Natural continuous extensions of {R}unge--{K}utta methods.
\newblock {\em Mathematics of Computation}, 46(173):119--133, 1986.

\end{thebibliography}
\end{document}